\documentclass[12pt,a4paper,oneside]{amsart}
\usepackage[a4paper]{geometry}
\usepackage{mathptmx} 
\DeclareMathAlphabet{\mathcal}{OMS}{cmsy}{m}{n} 
 
\usepackage{amssymb,mathrsfs}
\usepackage[UglyObsolete]{diagrams} 

\newcommand{\vertbar}{\>|\>}
\newcommand{\set}[2]{\ensuremath{\{ #1 \vertbar #2 \}}}
\newcommand{\modulus}[1]{\ensuremath{|\, #1 \,|}}

\DeclareMathOperator{\Con}{Con}
\DeclareMathOperator{\Ker}{Ker}

\newtheorem{theorem}{Theorem}
\newtheorem*{corollary}{Corollary}

\theoremstyle{definition}

\begin{document}

\title{On the utility of Robinson--Amitsur ultrafilters. II}
\author{Pasha Zusmanovich}
\address{Department of Mathematics, University of Ostrava, Ostrava, Czech Republic
}
\email{pasha.zusmanovich@osu.cz} 
\date{last revised August 14, 2016}
\thanks{J. Algebra, to appear; arXiv:1508.07496}

\begin{abstract}
We extend the result of the previous paper under the same title about embedding
of ideal-determined algebraic systems into ultraproducts, to arbitrary algebraic
systems, and to ultraproducts over $\kappa$-complete ultrafilters. We also 
discuss the scope of applicability of this result, and correct a mistake from 
the previous paper concerning homomorphisms from ultraproducts.
\end{abstract}

\maketitle

\subsection{Introduction}
In \cite{ultra}, an old and simple trick, used by A. Robinson and S. Amitsur
(see \cite[Proof of Theorem 15]{amitsur-1} and \cite[Theorem 3]{amitsur-2}) to 
establish, in the context of ring theory, an embedding of certain rings in ultraproducts, was generalized to algebraic systems with ideal-determined 
congruences. Here we push it to the full generality, for arbitrary algebraic 
systems without any restrictions on their congruences.

\subsection{Recollection on congruences, ultrafilters, and ultraproducts}
We refer to \cite{bergman-c} and \cite{chang-keisler} for the rudiments of
universal algebra and model theory we need.

Let us recall some basic notions and fix notation. We consider the most general
algebraic systems, i.e. sets with a number of operations $\Omega$ defined on 
them, of, generally, various arity. The signature $\Omega$ is arbitrary, but 
fixed (so, in what follows all algebraic systems are supposed to be of signature
$\Omega$). The set of all congruences on a given algebraic system $A$ forms a 
partially ordered set (actually, a lattice), and the minimal element in this set
is the \emph{trivial congruence}, coinciding with the diagonal 
$\set{(a,a)}{a\in A}$ in $A \times A$. The intersection of all congruences 
containing a given relation $\rho$, i.e. a set of pairs of elements from 
$A \times A$, is called a \emph{congruence generated by $\rho$} and is denoted 
by $\Con(\rho)$. In the case of one-element relation, i.e. a single pair 
$(a,b) \in A \times A$ with $a\ne b$, we shorten this notation to $\Con(a,b)$ 
and speak about \emph{principal congruences}.

Given a cardinal $\kappa > 2$, let us call a set $\mathscr S$ of sets 
\emph{$\kappa$-complete}, if the intersection of any nonempty set of fewer than
$\kappa$ elements of $\mathscr S$ belongs to $\mathscr S$. If a set $\mathscr S$
of subsets of a set $\mathbb I$ satisfies a weaker condition -- that the 
intersection of any nonempty set of fewer than $\kappa$ elements of $\mathscr S$
contains an element of $\mathscr S$ -- then the set of subsets of $\mathbb I$ 
which are oversets of all such intersections is a $\kappa$-complete filter on 
$\mathbb I$ containing $\mathscr S$ (this is an obvious generalization of the 
standard and frequently employed fact that any set of subsets satisfying the 
finite intersection property can be extended to a filter).

An algebraic system $A$ is called \emph{$\kappa$-subdirectly irreducible}, if 
either of the following two equivalent conditions is satisfied:
\begin{itemize}
\item 
The set of nontrivial congruences of $A$ is $\kappa$-complete.
\item
If $A$ embeds in the direct product of $< \kappa$ algebraic systems 
$\prod_{i\in \mathbb I} B_i$, $\modulus{\mathbb I} < \kappa$, then $A$ embeds in
one of $B_i$'s. 
\end{itemize}
Obviously, the condition of $\omega$-subdirect irreducibility, also called
\emph{finite subdirect irreducibility}, is equivalent to $n$-subdirect 
irreducibility for any finite $n>2$. Note that $\omega$-complete filters 
(ultrafilters) are just the usual filters (ultrafilters).

Given a filter (respectively, ultrafilter) $\mathscr F$ on a set $\mathbb I$,
the quotient of the direct product of algebraic systems 
$\prod_{i\in \mathbb I} A_i$ by the congruence 
$$
\theta_{\mathscr F} = 
\set{(a,b) \in (\prod_{i\in \mathbb I} A_i) \times (\prod_{i\in \mathbb I} A_i)}{\set{i\in \mathbb I}{a(i) = b(i)} \in \mathscr F}
$$
is called \emph{filtered product} (respectively, \emph{ultraproduct}) of the 
corresponding algebraic systems, and is denoted by $\prod_{\mathscr F} A_i$.

As any $\kappa$-complete ultrafilter on a set of cardinality $< \kappa$ is 
principal, the condition of $\kappa$-subdirect irreducibility of an algebraic 
system $A$ may be trivially reformulated as follows: 
\begin{itemize}
\item
If $A$ embeds in the direct product of $< \kappa$ algebraic systems 
$\prod_{i\in \mathbb I} B_i$, $\modulus{\mathbb I} < \kappa$, then $A$ embeds in
the ultraproduct $\prod_{\mathscr U} B_i$ for some ultrafilter $\mathscr U$ 
on $\mathbb I$.
\end{itemize}
What is, perhaps, surprising, is that under a suitable set-theoretic assumption
on the cardinal $\kappa$, the latter condition is equivalent to the same 
condition with an arbitrary index set $\mathbb I$, without a restriction on its cardinality. This equivalence is 
what constitutes the ``Robinson--Amitsur theorem'' (Theorem \ref{th-rob-am} 
below).

\subsection{Embeddings in direct products}\label{ss-th}

The next proof follows the proof of Theorem 1.1 from \cite{ultra}, with ideals 
being replaced by congruences, and $\omega$ being replaced by an arbitrary 
cardinal $\kappa$. On this level of generality, the proof is even simpler than
an already simple proof from \cite{ultra}.

\begin{theorem}[\sc Robinson--Amitsur]\label{th-rob-am}
Let $\kappa$ be a cardinal $>2$ and such that any $\kappa$-complete filter can 
be extended to a $\kappa$-complete ultrafilter. Then for any algebraic system 
$A$ the following are equivalent:
\begin{enumerate}
\item $A$ is $\kappa$-subdirectly irreducible.
\item
For any embedding $f$ of $A$ in the direct product $\prod_{i\in \mathbb I} B_i$
of a set of algebraic systems $\{B_i\}_{i\in \mathbb I}$, there is a 
$\kappa$-complete ultrafilter $\mathscr U$ on the set $\mathbb I$ such that the
composition of $f$ with the canonical homomorphism 
$\prod_{i\in \mathbb I} B_i \to \prod_{\mathscr U} B_i$, is an embedding.
\end{enumerate}
\end{theorem}

\begin{proof}
(ii) $\Rightarrow$ (i): 
take $\modulus{\mathbb I} < \kappa$ and observe, as above, that then the 
ultrafilter $\mathscr U$ is principal.

(i) $\Rightarrow$ (ii). For any two elements $a,b \in A$, define 
$\mathbb S_{a,b} = \set{i\in \mathbb I}{a(i) \ne b(i)}$, and let 
$\mathscr S = \set{\mathbb S_{a,b}}{a,b \in A, a \ne b}$. Let us verify that the
intersection of $< \kappa$ elements of $\mathscr S$ contains an element of 
$\mathscr S$.

Let $a,b$ be two elements of $A$ such that $a\ne b$. If for some 
$i_0\in \mathbb I$ we have $a(i_0) = b(i_0)$, then $(a,b)$ belongs to the 
congruence of $\prod_{i\in \mathbb I} B_i$ which is the kernel of the canonical
projection $\Pr_{i_0}$ to $B_{i_0}$, and by the Third Isomorphism Theorem, 
belongs to the congruence $\Ker(\Pr_{i_0}) \cap (A \times A)$ of $A$, whence 
$\Con(a,b) \subseteq \Ker(\Pr_{i_0})$. The latter means that for any 
$(c,d) \in \Con(a,b)$, we have $\mathbb S_{c,d} \subseteq \mathbb S_{a,b}$.

Due to $\kappa$-subdirect irreducibility of $A$, for any set of pairs of 
different elements $\{a_i, b_i\}$ of $A$ of cardinality $< \kappa$, the 
intersection $\bigcap_{i<\kappa} \Con(a_i,b_i)$ is a nontrivial congruence. Pick
an element $(c,d)$, $c\ne d$, from this intersection. Then by just proved we 
have $\mathbb S_{c,d} \subseteq \bigcap_{i<\kappa} \mathbb S_{a_i,b_i}$, as 
required.

Thus $\mathscr S$ can be extended to a $\kappa$-complete filter, and then to a
$\kappa$-complete ultrafilter $\mathscr U$ on $\mathbb I$. Factoring the 
embedding of $A$ in $\prod_{i\in \mathbb I} B_i$ by the congruence 
$\theta_\mathscr U$, we get, again by the Third Isomorphism Theorem, an 
embedding of the quotient $A / (\theta_{\mathscr U} \cap (A \times A))$ in the 
ultraproduct $\prod_{\mathscr U} B_i$. But if 
$(a,b) \in \theta_{\mathscr U} \cap (A \times A)$, then, by definition, 
$\set{i\in \mathbb I}{a(i) = b(i)} \in \mathscr U$, and, since $\mathscr U$ is
an ultrafilter, $\set{i\in \mathbb I}{a(i) \ne b(i)} \notin \mathscr U$, and
hence $\set{i\in \mathbb I}{a(i) \ne b(i)} \notin \mathscr S$, what, in its 
turn, implies $a=b$. This shows that the congruence 
$\theta_{\mathscr U} \cap (A \times A)$ is trivial, and hence $A$ embeds in
$\prod_{\mathscr U} B_i$.
\end{proof}

Coupling the case $\kappa = \omega$ of this theorem with the Birkhoff theorem 
about varieties of algebraic systems, and some rudimentary model theory, we get:

\begin{corollary}[\sc Criterion for absence of nontrivial identities]
Let $\mathfrak A$ be a variety of algebraic systems such that any free system
in $\mathfrak A$ is finitely subdirectly irreducible. Then for an algebraic 
system $A \in \mathfrak A$ the following three conditions are equivalent:
\begin{enumerate}
\item $A$ does not satisfy a nontrivial identity within $\mathfrak A$.
\item Any free system in $\mathfrak A$ embeds in an ultrapower of $A$.
\item Any free system in $\mathfrak A$ embeds in a filtered power of $A$.
\end{enumerate}
\end{corollary}

\begin{proof}
(i) $\Rightarrow$ (ii) is established exactly as in 
\cite[Corollaries 1.2 and 1.3]{ultra}: by the Birkhoff theorem, a free system in
$\mathfrak A$ embeds in the direct power of $A$. Apply Theorem \ref{th-rob-am},
with $\kappa = \omega$.

(ii) $\Rightarrow$ (iii): obvious.

(iii) $\Rightarrow$ (i): follows from the fact that any nontrivial identity is a
first-order Horn sentence, and Horn sentences are preserved under filtered 
products (see, e.g., \cite[Proposition 6.2.2]{chang-keisler}).
\end{proof}

\subsection{No way to move beyond $\omega$. Discussion}

Let us make a few remarks about the scope of applicability of 
Theorem \ref{th-rob-am} and Corollary. By the {\L}o\'s theorem, an algebraic 
system is elementarily equivalent to its ultrapower, so if the condition of 
Corollary is met, it allows to reduce the question about the absence of 
nontrivial identities in an algebraic system $A$ to the question whether a 
system elementary equivalent to $A$ contains a free subsystem. This was 
demonstrated to be useful in some situations in \cite{ultra}, but the 
usefulness is severely restricted by the fact that any property of an algebraic
system $A$ we can use in the process should be a first-order property. Most of 
the interesting properties of algebraic systems are second- or higher-order. 

Note that by the analog of {\L}o\'s theorem for ultraproducts over 
$\kappa$-complete ultrafilters, the ultrapower $A^{\mathscr U}$ of an algebraic
system $A$ is elementarily equivalent to $A$ in the sense of certain 
higher-order logic (denoted by $\mathbf{L}_\kappa$ in 
\cite[\S 4.2]{chang-keisler}, and by $\mathbf{L}_{\kappa\kappa}$ in 
\cite[\S 3.3]{dickmann}). But unfortunately, we cannot derive an analog of 
Corollary to Theorem \ref{th-rob-am} for $\kappa > \omega$, at least in, 
arguably, the most interesting for applications cases -- groups and algebras 
over fields (the main protagonists of \cite{ultra}). The reason is that free 
systems in these varieties are residually nilpotent (the intersection of all 
terms of the lower central series is trivial) and hence are already not 
$\omega_1$-subdirectly irreducible. 

Moreover, it seems that no statement similar to Corollary, where the (usual) 
ultrapower construction is replaced by whatever other construction providing us
with elementary equivalence in the sense of second- or higher-order logic of 
some sort, is possible. Indeed, any such logic \textbf{L} should be a priori 
weak enough not to allow encode the fact that an algebraic system contains a 
free subsystem: otherwise, if in the condition (ii) of Corollary we would be 
able to replace the ultrapower of $A$ by another construction elementary equivalent to 
$A$ in the sense of \textbf{L}, this would imply that for any algebraic system 
the absence of a nontrivial identity is equivalent to having a free subsystem --
an obviously false statement. While, in general, free systems in arbitrary 
varieties are not necessary countable, they are, essentially, so in all cases of
interest (groups and algebras over fields): all identities in a given variety 
are captured by free systems of no more than countable rank; such free groups 
are, obviously, countable, and such free algebras are countable-dimensional over
the base field. In the latter case, as identities of algebras do not change 
under field extensions (except, possibly, the cases of finite base fields of 
``small'' cardinality), we may always assume the base field to be countable, 
and, hence, the whole free algebra to be countable. To summarize: the 
hypothetical logic \textbf{L} should not allow to encode the fact that an 
algebraic system contains a countable subsystem -- a very weak logic indeed, on
the verge not to be qualified as the ``second-order'' one!

What about the condition on $\kappa$ in Theorem \ref{th-rob-am} (sometimes 
called \emph{strong compactness})? In the case $\kappa = \omega$ this amounts
to the well-known fact that any filter can be extended to an ultrafilter (true 
by the Zorn lemma). In the general case we are in deep waters of set theory. It
seems that this is one of the statements about which one does not make sense 
to speak if they are ``true'' or ``false'', but it is rather a matter of which 
model of set theory is adopted. See \cite[\S 3.3, Parts C,D]{dickmann} for this
and many other conditions of such sort, and relationships between them. Note 
also that the very existence of non-principal $\kappa$-complete ultrafilters for
$\kappa > \omega$, being equivalent to the existence of measurable cardinals, 
entails rather strange properties of the corresponding model(s) of set theory 
(see, again, \cite[\S 4.2]{chang-keisler} and \cite[\S 0.4]{dickmann}).

Note finally that it would be very interesting to extend arguments used in the
proof of Theorem \ref{th-rob-am} to the case of metric ultraproducts. This could
provide an approach to various questions related to sofic and hyperlinear 
groups.

\subsection{Semigroups}

Arguably, the most interesting class of algebraic systems whose congruences are
not ide\-al-de\-ter\-mi\-ned, i.e. not covered by \cite{ultra}, but by 
Theorem \ref{th-rob-am} and Corollary above, is the class of semigroups. 
Unfortunately, the Corollary is not applicable to the variety of all semigroups,
as free semigroups of rank $>1$ are not finitely subdirectly irreducible. For 
example, if $x$ and $y$ are among free generators of a free semigroup $G$, then
the congruences of $G$ generated by pairs $(x,x^2)$ and $(y,y^2)$ intersect 
trivially. It seems to be interesting to investigate for which classes of 
semigroups (inverse semigroups? Burnside varieties?) the Corollary would be 
applicable, and what good it could do. The same question for the class of loops.

\subsection{Homomorphisms from direct products}

Finally, we take the opportunity to note a mistake in \cite{ultra}. Theorem 7.1
there contains a statement dual, in a sense, to Theorem \ref{th-rob-am}: if 
there is a surjective homomorphism $\alpha: \prod_{i\in \mathbb I} B_i \to A$, 
where $A$ is finitely subdirectly irreducible, then $A$ is a surjective 
homomorphic image of an ultraproduct of the same family $\{B_i\}$. The proof of
this theorem is in error. At the top of p.~283 in the published version (at the
bottom of p.~13--top of p.~14 in the arXiv version), it is, essentially, claimed
(keeping the notation and terminology of the original paper) that any equality 
in $A$ of the form $\alpha(u) = t(b_1,\dots,b_n,\alpha(f),\dots,\alpha(f))$, 
where $u,f\in \prod_{i\in \mathbb I} B_i$, $b_1, \dots, b_n \in A$, and $t$ is 
an ideal term in arguments occupied by $\alpha(f)$'s, can be lifted to 
$\prod_{i\in \mathbb I} B_i$: $u = t(h_1,\dots,h_n,f,...,f)$ for some 
$h_i \in \prod_{i\in \mathbb I} B_i$ such that $\alpha(h_i) = b_i$. There is no
reason whatsoever for this to be true in general: for example, in the simplest 
possible case $t(x) = x$ (the ideal term just in one variable, without arguments
occupied by $b_i$'s), this claim amounts to saying that $\alpha$ is an 
isomorphism.

Luckily, we can slightly modify the statement, and supply it with a new proof. 
While the new statement differs from the original one, it is more general for 
the most classes of algebraic systems of interest. The proof uses an ingenious 
(and barely utilizing anything beyond the definition of ultrafilter) argument 
from an array of recent papers by G.~Bergman, some of them jointly with 
N.~Nahlus (e.g., \cite{bergman-g} and \cite{bergman}), devoted to factoring homomorphisms from 
direct product of groups, algebras over fields, and modules. The argument is valid for arbitrary 
algebraic  systems, and we reproduce it here, throwing in $\kappa$-complete 
ultrafilters for good measure.

\begin{theorem}[\sc Bergman--Nahlus]\label{th-berg}
For any cardinal $\kappa > 2$, and any algebraic system $A$ consisting of more 
than one element, the following are equivalent:
\begin{enumerate}
\item
For any surjective homomorphism $f$ of the direct product 
$\prod_{i\in \mathbb I} B_i$, $\modulus{\mathbb I} < \kappa$, to $A$, there is 
$i_0 \in \mathbb I$ such that $f$ factors through the canonical projection 
$\prod_{i\in \mathbb I} B_i \to B_{i_0}$.
\item
For any surjective homomorphism $f$ of the direct product 
$\prod_{i\in \mathbb I} B_i$ to $A$, there is a $\kappa$-complete ultrafilter 
$\mathscr U$ on the indexing set $\mathbb I$ such that $f$ factors through the 
canonical homomorphism $\prod_{i\in \mathbb I} B_i \to \prod_{\mathscr U} B_i$.
\end{enumerate}
\end{theorem}

\begin{proof}
(ii) $\Rightarrow$ (i) is obvious (like in a similar situation in 
Theorem \ref{th-rob-am}, take $\modulus{\mathbb I} < \kappa$ and observe that 
$\mathscr U$ is necessary principal).

(i) $\Rightarrow$ (ii).
Let $\mathscr U$ consist of subsets $\mathbb J \subseteq \mathbb I$ such that
$f$ factors through a surjective homomorphism from the (smaller) direct product
of $B_i$'s indexed over $\mathbb J$, i.e. the diagram
\begin{diagram}
\prod_{i\in \mathbb I} B_i & \rTo^f  & A   \\
\dTo                       & \ruTo         \\ 
\prod_{i\in \mathbb J} B_i &  
\end{diagram}
where the vertical arrow is the canonical projection, commutes. Utilizing the 
condition (i) with $\modulus{\mathbb I} = 2$, and reasoning as in 
\cite[Lemma 7 and Proposition 8]{bergman-g} or \cite[Lemma 1.2]{bergman}, we get
that $\mathscr U$ is an ultrafilter, and $f$ factors through the canonical 
homomorphism $\prod_{i\in \mathbb I} B_i \to \prod_{\mathscr U} B_i$.

To prove that $\mathscr U$ is $\kappa$-complete, it is sufficient to show that
for any decomposition $\mathbb I = \bigcup_{j < \kappa} \mathbb I_j$ into the 
union of pairwise disjoint sets $\mathbb I_j$, $j<\kappa$, at least one of them
belongs to $\mathscr U$. Since 
$\prod_{i\in \mathbb I} B_i = 
\prod_{j<\kappa} \Big(\prod_{i\in \mathbb I_j} B_i\Big)$, $f$ factors through
the canonical projection 
$\prod_{i\in \mathbb I} B_i \to \prod_{i\in \mathbb I_{j_0}} B_i$ for some 
$j_0 < \kappa$, i.e. $\mathbb I_{j_0} \in \mathscr U$, as required.
\end{proof}

Note that in the case of groups and algebras over fields, the condition (i) 
of Theorem \ref{th-berg} with $\kappa = \omega$ is weaker than the finite 
subdirect irreducibility. This follows from \cite[Lemma 5]{bergman-g} in the 
case of algebras over fields, and in the case of groups the argument is repeated
almost verbatim. 

Note also that though Theorems \ref{th-rob-am} and \ref{th-berg} are, in a 
sense, dual to each other, this duality does not stretch to the proofs: the 
proofs are different, and each proof seemingly cannot be adapted to the dual
situation.

\subsection*{Acknowledgements}

Thanks are due to Alexander Gein, Fares Maalouf, and Eugene Plotkin for very 
useful remarks: Gein has dispelled my illusion about well-behavior of semigroup
congruences, Maalouf has turned my attention to a mistake in \cite{ultra}, and
Plotkin has insisted, a long time ago, that the setup of \cite{ultra} can be 
generalized further (it took me several years to realize this -- obvious, as it
looks now --  fact).
The financial support of the Regional Authority of the Moravian-Silesian Region
(grant MSK 44/3316) and of the Ministry of Education and Science of the 
Republic of Kazakhstan (grant 0828/GF4) is gratefully acknowledged.

\end{document}